\documentclass[12pt,a4paper]{article}
\usepackage{amsmath,amssymb,amsfonts, amsthm}

\title{\bf Dedekind Sums with Arguments near Certain Transcendental Numbers}

\author{Kurt Girstmair\\
Institut f\"ur Mathematik \\
Universit\"at Innsbruck   \\
Technikerstr. 13/7        \\
A-6020 Innsbruck, Austria \\
Kurt.Girstmair@uibk.ac.at}
\date{}

\makeatletter
\let\@@maketitle=\maketitle
\def\maketitle{\def\thispagestyle##1{\relax}\@@maketitle}
\makeatother
\newtheorem{theorem}{Theorem}
\newtheorem{lemma}{Lemma}
\newtheorem{prop}{Proposition}
\newenvironment{rem}{{\em Remark.}}{}

\def\BE{\begin{equation}}
\def\EE{\end{equation}}
\def\BD{\begin{displaymath}}
\def\ED{\end{displaymath}}
\def\BEA{\begin{eqnarray*}}
\def\EEA{\end{eqnarray*}}
\def\BI{\bibitem}

\def\Z{\mathbb Z}
\def\Q{\mathbb Q}

\def\MB{\mbox}
\def\LD{\ldots}

\def\DED{Dedekind }
\def\WH{\widehat}

\begin{document}

\maketitle

\begin{abstract}

We study the asymptotic behaviour of the classical Dedekind sums $s(s_k/t_k)$ for the sequence of convergents $s_k/t_k$
$k\ge 0$, of the transcendental number
\BD
   \sum_{j=0}^\infty\frac 1{b^{2^j}},\ b\ge 3.
\ED
In particular, we show that there are infinitely many open intervals of constant length such that the
sequence $s(s_k/t_k)$ has infinitely many transcendental cluster points in each interval.

\end{abstract}

Continued fraction expansions of transcendental numbers

\noindent
{\em AMS Subject Classification:}
\\ {\em Primary:} 11 F 20,
\\ {\em Secondary:} 11 A 55

\newpage

%%%%%%%%%%%%%%%%%%%%%%%%%%%%%%%%%%%%%%
\section{Introduction and result}
%%%%%%%%%%%%%%%%%%%%%%%%%%%%%%%%%%%%%%

\DED sums have quite a number of interesting applications in
analytic number theory (modular forms), algebraic number theory (class numbers),
lattice point problems and algebraic geometry
(for instance \cite{{Ap}, {Me}, {RaGr}, {Ur}}).

Let $n$ be a positive integer and $m\in \Z$, $(m,n)=1$. The classical \DED sum $s(m/n)$ is defined by
\BD
   s(m/n)=\sum_{k=1}^n ((k/n))((mk/n))
\ED
where $((\LD))$ is the usual sawtooth function (for example, \cite[p.\ 1]{RaGr}).
In the present setting it is more
natural to work with
\BD
S(m/n)=12s(m/n)
\ED instead.

In the previous paper \cite{Gi} we used the Barkan-Hickerson-Knuth-formula to study the asymptotic behaviour
of $S(s_k/t_k)$ for the convergents $s_k/t_k$ of transcendental numbers like $e$ or $e^2$.
In this situation the limiting behaviour of $S(s_k/t_k)$ was fairly simple. It is much more complicated, however,
for the transcendental number
\BE
\label{0.2}
   x(b)=\sum_{j=0}^\infty\frac 1{b^{2^j}},\ b\ge 3.
\EE
In fact, we have no full description of what happens in this case. Its complexity is illustrated by the following
theorem, which forms the main result of this paper.

\begin{theorem} % Theorem 1 %%%%%%%%%%%%%%%%%%%%%%%%%%%%%%%%%%%%%%%%%%%
\label{t1}
Let $s_k/t_k$, $k\ge 0$, be the sequence of convergents of the number $x(b)$ of {\rm (\ref{0.2})}.
Then the sequence $S(s_k/t_k)$, $k\ge 0$, has infinitely many transcendental cluster points
in each of the intervals
\BD
  \left(b-10-2i+\frac 1b,b-9-2i+\frac 1{b-1}\right),\ i\ge 0.
\ED
\end{theorem} %%%%%%%%%%%%%%%%%%%%%%%%%%%%%%%%%%%%%%%%%%%%%%%%%%%%%%%%

Note that each of the intervals of Theorem \ref{t1} has the length $1+1/(b(b-1))$, whereas the distance
between two neighbouring intervals is $1-1/(b(b-1))$.

%%%%%%%%%%%%%%%%%%%%%%%%%%%%%%%%%%%%%%
\section{The integer part}
%%%%%%%%%%%%%%%%%%%%%%%%%%%%%%%%%%%%%%

We start with the continued fraction expansion $[a_0,a_1,a_2,\LD]$ of an arbitrary irrational number $x$.
The numerators and denominators of its convergents
\BE
\label{2.0.1}
s_k/t_k=[a_0,a_1,\LD,a_k]
\EE
are defined by the recursion formulas
\begin{eqnarray}
\label{2.0}
   s_{-2}=0,&& s_{-1}=1,\hspace{5mm} s_k=a_ks_{k-1}+s_{k-2} \,\MB{ and } \nonumber\\
   t_{-2}=1,&& t_{-1}=0,\hspace{5mm} t_k=a_kt_{k-1}+t_{k-2},\,\MB{ for } k\ge 0.
\end{eqnarray}
Henceforth we will assume $0<x<1$, so $a_0=0$.
Then the Barkan-Hickerson-Knuth formula says that for $k\ge 0$
\BE
\label{2.2}
S(s_k/t_k)=\sum_{j=1}^k(-1)^{j-1}a_j+\begin{cases} \vspace{5mm}
                                                              (s_k+ t_{k-1})/{t_k}-3, & \MB{if } k \MB{ is odd}; \\
                                                               (s_k-t_{k-1})/{t_k},  & \MB{if } k \MB{ is even}
                                      \end{cases}
\EE
(see \cite{Ba}, \cite{Hi}, \cite{Kn}).

In the case of the number $x=x(b)$, the continued fraction expansion has been given in \cite{Sh}.
It is defined recursively. To this end
put
\BD
C(1)=C(1,b)=[0,b-1,b+2]
\ED
in the sense of (\ref{2.0.1}) and (\ref{2.0}).
If $C(j)=C(j,b)$ has been defined for $j\ge 1$ and $C(j)=[0,a_1,\LD,a_n]$ (where $n=2^j$),
then
\BD
   C(j+1)=C(j+1,b)=[0,a_1,\LD,a_n,a_n-2,a_{n-1},a_{n-2},\LD,a_2,a_1+1].
\ED
Then $x=\lim_{j\to\infty}C(j)$. In particular, $x=[0,a_1,a_2,\LD]$, where $a_k$ is the corresponding partial
denominator of each $C(j)$ with $2^j\ge k$.

In view of formula (\ref{2.2}) for $x=x(b)$, it is natural to investigate
\BD
  L(k)=L(k, b)=\sum_{j=1}^{k-1}(-1)^{j-1}a_j,\ k\ge 0,
\ED
first. For the sake of simplicity we call $L(k)$ the {\em integer part} of the \DED sum $S(s_k/t_k)$.

The following lemma comprises three easy observations.

\begin{lemma} % Lemma 1 %%%%%%%%%%%%%%%%%%%%%%%%%%%%%%%%%%%%%%%%%%%
\label{l1}
Let $[0,a_1,a_2,\LD]$ be the continued fraction expansion of $x=x(b)$ and $n=2^j$, $j\ge 0$.

  {\rm (a)} If $n\ge 4$, then
\BD
  a_{n+k}=a_{n-k+1} \MB{ for } 2\le k\le n-1.
\ED

  {\rm (b)} If $n\ge 8$, then
\BD
   a_k=a_{n-k+1} \MB{ for } 2\le k\le n/2-1.
\ED

  {\rm (c)} If $n\ge 8$, then
\BD
  a_k=a_{n+k} \MB{ for } 2\le k\le n/2-1.
\ED
\end{lemma} %%%%%%%%%%%%%%%%%%%%%%%%%%%%%%%%%%%%%%%%%%%%%%%%%%%%%%%%%

\begin{proof}
Obviously, assertion (c) follows from (a) and (b).
Assertion (a) is immediate from the definition of the continued fraction expansion of $x(b)$.
In order to deduce (b) from (a), we assume $n\ge 4$ and put $l=n-k+1$, $2\le k\le n-1$.
Then $a_l=a_{n-k+1}=a_{n+k}$, by (a). Since $k=n-l+1$, this gives $a_l=a_{n+(n-l+1)}=a_{2n-l+1}$.
So we have, for $n\ge 8$ and $2\le l\le n/2-1$: $a_l=a_{n-l+1}$, which is (b).
\end{proof}

\begin{lemma} % Lemma 2 %%%%%%%%%%%%%%%%%%%%%%%%%%%%%%%%%%%%%%%%%%%%%%
\label{l2} Let $n=2^j$, $n\ge 4$. For $1\le k\le n-1$ we have
\BD
  L(n+k)=-2+L(n-k).
\ED
\end{lemma} %%%%%%%%%%%%%%%%%%%%%%%%%%%%%%%%%%%%%%%%%%%%%%%%%%%%%%%%%%%

\begin{proof}
Since $L(n+1)=L(n-1)+(-1)^{n-1}a_n+(-1)^n(a_n-2)=L(n-1)-2$, the assertion holds for $k=1$.
Let $2\le k\le n-1$. Then
\BD
 L(n+k)=L(n-1)-2+\sum_{i=2}^k (-1)^{n+i-1}a_{n+i}.
\ED
By assertion (a) of Lemma \ref{l1}, the sum on the right hand side equals
\BD
  \sum_{i=2}^k(-1)^{n+i-1}a_{n-i+1}=\sum_{i=2}^k(-1)^{i-1}a_{n-i+1}=\sum_{i=1}^{k-1}(-1)^ia_{n-i}.
\ED
We observe
\BD
  \sum_{i=1}^{k-1}(-1)^ia_{n-i}=\sum_{i=n-k+1}^{n-1}(-1)^ia_i.
\ED
This gives
\BD
 L(n+k)=-2+\sum_{i=1}^{n-1}(-1)^{i-1}a_i+\sum_{i=n-k+1}^{n-1}(-1)^ia_i=-2+L(n-k).
\ED
\end{proof}

\noindent
\begin{rem}
 By the construction of the sequence $C(j)$, we have $a_n=b$ for each $n=2^j, j\ge 2$.
 From Lemma \ref{l2} we obtain
 $
   L(2n)=L(2n-1)+(-1)^{2n-1}a_{2n}=L(n+(n-1))-b=L(1)-2-b=b-1-2-b=-3.
 $
\end{rem}

\begin{lemma} % Lemma 3 %%%%%%%%%%%%%%%%%%%%%%%%%%%%%%%%%%%%%%%%%%%%%%%
\label{l3}
Let $n=2^j$, $n\ge 8$.
For $2\le k\le n/2-1$,
\BD
 L(n+k)=-4+L(k).
\ED
In particular, $L(n+k)=L(2n+k)=L(4n+k)=\LD$
\end{lemma} %%%%%%%%%%%%%%%%%%%%%%%%%%%%%%%%%%%%%%%%%%%%%%%%%%%%%%%%%%%

\begin{proof}
We have $L(n)=-3$, by the remark. Hence $L(n+1)=L(n)+(-1)^na_{n+1}=-3+b-2=b-5$.
From Lemma \ref{l1}, (c) we obtain
\BD
  L(n+k)=b-5+(-1)^{n+1}a_{n+2}+\LD+(-1)^{n+k-1}a_{n+k}=
\ED
\BD
  b-5+(-1)^1a_2+\LD+(-1)^{k-1}a_k=b-5+L(k)-a_1=-4+L(k).
\ED
\end{proof}

Let $n=2^j$, $n\ge 8$. We define a sequence $k_i$, $i\ge 0$, in the following way:
\BE
\label{2.4}
  k_0=n-1.
\EE
If $k_{i-1}$ has been defined, $i\ge 1$, then
\BE
\label{2.6}
 k_{i}=2^{i}n-k_{i-1}.
\EE
Induction based on (\ref{2.4}) and (\ref{2.6}) gives
\BE
\label{2.7}
  2\le k_i\le 2^in-1,
\EE
and
\BE
\label{2.8}
k_i=\frac{2^{i+1}+(-1)^i}{3}\,n+(-1)^{i-1}
\EE
for all $i\ge 0$.
We have
\BD
 L(k_0)=L(n-1)=L(n)+a_n=-3+b
\ED
from the remark. Further, Lemma \ref{l2} gives, by induction,
\BD
\label{2.10}
L(k_i)=-3-2i+b.
\ED
Indeed, if $L(k_{i-1})=-3-2(i-1)+b$,
$L(k_i)=L(2^in-k_{i-1})=L(2^{i-1}n+(2^{i-1}n-k_{i-1}))=-2+L(k_{i-1})=-3-2i+b.$
Altogether, we know the numbers $k_i$ and the integer part of $S(s_{k_i}/t_{k_i})$
explicitly, namely

\begin{lemma} % Lemma 4 %%%%%%%%%%%%%%%%%%%%%%%%%%%%%%%%%%%%%%%%%%%%%%
\label{l4}
Let $n=2^j$, $n\ge 8$. For $i\ge 0$ let $k_i$ be defined by {\rm(\ref{2.8})}. Then
\BD
 L(k_i)=b-3-2i.
\ED
\end{lemma} %%%%%%%%%%%%%%%%%%%%%%%%%%%%%%%%%%%%%%%%%%%%%%%%%%%%%%%%%%%

Lemma \ref{l4} says that the integer part $L(k_i)$ of $S(s_{k_i}/t_{k_i})$ is independent of $n$
if $n\ge 8$ is a power of $2$. Suppose, therefore, that $n_l=2^{2+l}$, $l=1,\LD,r$. Fix $i\ge 0$ for
the time being and define
\BE
\label{2.12}
k_{i,l}=\frac{2^{i+1}+(-1)^i}{3}n_l+(-1)^{i-1}.
\EE
By (\ref{2.7}),
\BD
  k_{i,l}\le 2^in_l-1\le 2^in_r-1=2^{i+r+2}-1.
\ED
Suppose that $\WH n$ is a power of $2$, $\WH n\ge 2^{i+r+3}$.
Then we have
\BD
   2\le k_{i,l}\le\frac{\WH n}{2}-1
\ED
for all $l=1,\LD,r$. Therefore, Lemma \ref{l3} and Lemma \ref{l4}
give

\begin{prop} % Proposition 1 %%%%%%%%%%%%%%%%%%%%%%%%%%%%%%%%%%%%%%%%%%%%%%%%%
\label{p1}
Let $i\ge 0$ and $r\ge 1$ be given and $n_l=2^{2+l}$, $l=1,\LD r$.  Suppose
that the numbers $k_{i,l}$ are defined as in {\rm (\ref{2.12})}. If $\WH n$ is a power
of $2$, $\WH n\ge 2^{i+r+3}$, then
\BD
  L(\WH n+k_{i,l})=-4+L(k_{i,l})=b-7-2i.
\ED
\end{prop} %%%%%%%%%%%%%%%%%%%%%%%%%%%%%%%%%%%%%%%%%%%%%%%%%%%%%%%%%%%%%

%%%%%%%%%%%%%%%%%%%%%%%%%%%%%%%%%%%%%%
\section{The fractional part}
%%%%%%%%%%%%%%%%%%%%%%%%%%%%%%%%%%%%%%

Note that the numbers $k_{i,l}$ of the foregoing section are all odd. Hence
Lemma \ref{l5} and the Barkan-Hickerson-Knuth formula give
\BE
\label{3.2}
 S(s_{\WH n+k_{i,l}}/ t_{\WH n+k_{i,l}})=b-7-2i +\frac{s_{\WH n+k_{i,l}}}{ t_{\WH n+k_{i,l}}}+
 \frac{t_{\WH n+k_{i,l-1}}}{t_{\WH n+k_{i,l}}}-3.
\EE
If $\WH n$ tends to infinity $s_{\WH n+k_{i,l}}/t_{\WH n+k_{i,l}}$ tends to $x=x(b)$. Accordingly,
we have to investigate the limiting behaviour of $t_{\WH n+k_{i,l-1}}/t_{\WH n+k_{i,l}}$ in order
to understand the fractional part of formula (\ref{3.2}).

To this end we suppose
that $n$ is a power of $2$, $n\ge 8$, and $k$ is an integer, $2\le k\le n/2-1$. From (\ref{2.0})
we have $t_{n+k}=a_{n+k}t_{n+k-1}+t_{n+k-2}$, hence
\BD
   \frac{t_{n+k}}{t_{n+k-1}}= a_{n+k}+\frac{t_{n+k-2}}{t_{n+k-1}}=[a_{n+k},\frac{t_{n+k-1}}{t_{n+k-2}}].
\ED
When we repeat this procedure, we obtain the well-known fact
\BD
\label{3.4}
  \frac{t_{n+k}}{t_{n+k-1}}= [a_{n+k}, a_{n+k-1},\frac{t_{n+k-2}}{t_{n+k-3}}]=[a_{n+k},a_{n+k-1},\LD, a_1].
\ED
From Lemma \ref{l1}, (c), we infer
\BD
   a_{n+k}=a_k, a_{n+k-1}=a_{k-1}, \LD ,  a_{n+2}=a_2.
\ED
Moreover, $a_{n+1}=a_n-2=b-2$ and $a_n=b$. Finally, Lemma \ref{l1}, (b) says
\BD
  a_{n-1}=a_2,a_{n-2}=a_3,\LD,a_{n/2+2}=a_{n/2-1}.
\ED
Altogether,
\BD
  \frac{t_{n+k}}{t_{n+k-1}}= [a_{k},a_{k-1},\LD, a_2,b-2,b,a_2,a_3,\LD a_{n/2-1},a_{n/2+1},\LD,a_1].
\ED
The final terms $a_{n/2+1}, a_{n/2}, \LD,a_1$ are not of interest. It suffices to write
\BE
\label{3.6}
  \frac{t_{n+k}}{t_{n+k-1}}= [a_{k},a_{k-1},\LD, a_2,b-2,b,a_2,a_3,\LD a_{n/2-1},c(n)]
\EE
for some $c(n)\in\Q$. From Theorem 8 in \cite{Sh} we know that all numbers $a_1, a_2,\LD $ are $\ge 1$ and $\le b+2$,
hence we have
\BD
  1\le c(n)\le b+3.
\ED

\begin{prop} % Proposition 2 %%%%%%%%%%%%%%%%%%%%%%%%%%%%%%%%%%%%%%%%%%%%%%%
\label{p2}
Suppose that $k$ remains fixed, $2\le k\le n/2-1$, but $n=2^j$ tends to infinity.
Then $t_{n+k}/t_{n+k-1}$ converges to
\BD
  t(k)=t(k,b)=[a_k,a_{k-1},\LD,a_2,b-2,(x+1)/x],
\ED
where $x=x(b)$ is defined by {\rm (\ref{0.2})}.
\end{prop} %%%%%%%%%%%%%%%%%%%%%%%%%%%%%%%%%%%%%%%%%%%%%%%%%%%%%%%%%%%

\begin{proof}
We have $x=\lim_{i\to\infty}C(i)=[0,b-1,y]$ with $y=[a_2,a_3,\LD]$. A short calculation shows
\BD
  [b, y]=[b,a_2,a_3,\LD]=(x+1)/x.
\ED
Let $p_i/q_i$, $i=0,1,2,\LD$ be the convergents of $t_{n+k}/t_{n+k-1}$ (where the numbers $p_i$, $q_i$
are defined in the same way as the numbers $s_i$, $t_i$ in (\ref{2.0})).
We have, by (\ref{3.6}),
\BD
  \frac{t_{n+k}}{t_{n+k-1}}=\frac{pc(n)+p'}{qc(n)+q'}
\ED
with $p=p_{k+n/2-1}$, $p'=p_{k+n/2-2}$, $q=q_{k+n/2-1}$, $q'=q_{k+n/2-2}$.
We write
\BD
  t(k)=[a_k,\LD,a_2,b-2,b,a_2,\LD,a_{n/2-1}, z(n)],
\ED
where $z(n)$ satisfies $1\le z(n)\le b+3$ by the argument above. Accordingly,
\BD
 t(k)=\frac{pz(n)+p'}{qz(n)+q'}.
\ED
This gives
\BE
\label{3.8}
   t(k)-\frac{t_{n+k}}{t_{n+k-1}}=\frac{pz(n)+p'}{qz(n)+q'}-\frac{pc(n)+p'}{qc(n)+q'}.
\EE
The expression on the right hand side of (\ref{3.8}) simplifies to
\BD
   \frac{(pq'-p'q)z(n)+(p'q-pq')c(n)}{(qz(n)+q')(qc(n)+q')}.
\ED
However, it is well-known that $pq'-p'q=\pm 1$. Observing $1\le z(n), c(n)\le b+3$,
we obtain
\BD
  \left|t(k)-\frac{t_{n+k}}{t_{n+k-1}}\right|\le \frac{2b+6}{(q+q')^2}.
\ED
Since $q$ and $q'$ tend to infinity for $n\to\infty$, our proof is complete.
\end{proof}

We conclude this section with two observations.

\begin{lemma} % Lemma 5 %%%%%%%%%%%%%%%%%%%%%%%%%%%%%%%%%%%%%%%%%%%%%%
\label{l5}
In the above setting, let $2\le k<k'$ be integers. Then $t(k)\ne t(k')$.
\end{lemma} %%%%%%%%%%%%%%%%%%%%%%%%%%%%%%%%%%%%%%%%%%%%%%%%%%%%%%%%%%

\begin{proof}
Suppose $t(k)=t(k')$, so
\BD
  [a_{k'},\LD a_{k+1},t(k)]=t(k).
\ED
An identity of this kind can only hold if $t(k)$ is a quadratic irrationality.
However, $t(k)$ is a transcendental number since $x$ is transcendental (see \cite[p. 35, Satz 8]{Sch}).
\end{proof}

\begin{lemma} % Lemma 6 %%%%%%%%%%%%%%%%%%%%%%%%%%%%%%%%%%%%%%%%%%%%%%
\label{l6}
Let $k\ge 2$ be an integer. Then $x+1/t(k)$ is a transcendental number.
\end{lemma} %%%%%%%%%%%%%%%%%%%%%%%%%%%%%%%%%%%%%%%%%%%%%%%%%%%%%%%%%%

\begin{proof}
Suppose $\alpha=x+t(k)$ is algebraic. Since we may write
\BD
   1/t(k)=[0,t(k)]=\frac{p(x+1)/x+p'}{q(x+1)/x+q'}=\frac{p(x+1)+p'x}{q(x+1)+q'x}
\ED
with integers $p,p',q,q'$, $q>0$, $q'\ge 0$,
we obtain
\BD
  x+\frac{p(x+1)+p'x}{q(x+1)+q'x}=\alpha.
\ED
This, however, means that $x$ satisfies a quadratic equation over the field $\Q(\alpha)$.
Accordingly, $x$ is algebraic, a contradiction.
\end{proof}

%%%%%%%%%%%%%%%%%%%%%%%%%%%%%%%%%%%%%%
\section{Proof of Theorem \ref{t1}}
%%%%%%%%%%%%%%%%%%%%%%%%%%%%%%%%%%%%%%

As in the setting of Proposition \ref{p1}, let $i\ge 0$ and $r\ge 1$ be given and $n_l=2^{2+l}$, $l=1,\LD, r$.
Suppose that the numbers $k_{i,l}$ are defined as in {\rm (\ref{2.12})}. Let $\WH n$ be a power
of $2$, $\WH n\ge 2^{i+r+3}$. By Proposition \ref{p1},
\BD
  L(\WH n+k_{i,l})=b-7-2i.
\ED
If $\WH n$ tends to infinity, Proposition \ref{p2} says that $t_{\WH n+k_{i,l}}/t_{\WH n+k_{i,l}-1}$
tends to
\BD
  t(k_{i,l})=[a_{k_{i,l}}, a_{k_{i,l}-1},\LD,a_2,b-2,(x+1)/x].
\ED
Therefore $t_{\WH n+k_{i,l}-1}/t_{\WH n+k_{i,l}}$ tends to $1/t(k_{i,l})$. Altogether, we have
\BD
  S(s_{\WH n+k_{i,l}}/t_{\WH n+k_{i,l}})\to b-10-2i+x+\frac{1}{t(k_{i,l})}.
\ED
For $l<l'\le r$ we obtain $k_{i,l}<k_{i,l'}$ from (\ref{2.12}).
By Lemma \ref{l5}, $t(k_{i,l})\ne t(k_{i,l'})$. Accordingly, the numbers $1/t(k_{i,l})$ are pairwise
different for $1\le l\le r$. Further, $x+1/t(k_{i,l})$ is transcendental, by Lemma \ref{l6}.
The inequalities
\BD
 1/b<x<1/(b-1) \MB{ and } 0<1/t(k_{i,l})<1
\ED
are obvious by (\ref{0.2}) and $x=[0,b-1,\LD]$, $1/t(k_{i,l})=[0, a_{k_{i,l}},\LD]$.
Therefore, the sequence $S(s_j/t_j)$, $j\ge 1$, has $r$ distinct transcendental cluster points in the interval
\BD
  \left(b-10-2i+\frac 1b,b-9-2i+\frac1{b-1}\right).
\ED
Since $r$ can be chosen arbitrarily large, this proves Theorem \ref{t1}.

%%%%%%%%%%%%%%%%%%%%%%%%%%%%%%%%%%%%%%%%%%%%%%%%%%%%%
%%%%%%%%%%%%%%%%%%%%%%%%%%%%%%%%%%%%%%%%%%%%%%%%%%%%%%%%%%%%%%%%%%%%%%%%%%

\end{document}